\documentclass[12pt,english]{amsart}

\input xy
\xyoption{all}

\usepackage{comment}
\usepackage[latin1]{inputenc}
\usepackage{amsfonts,amsmath,amssymb,amsthm}
\usepackage{url}

\newcommand{\Zz}{\mathbb{Z}}
\newcommand{\Qq}{\mathbb{Q}}

{\theoremstyle{plain}
\newtheorem{theorem}{Theorem}[section]    

\newtheorem{twisting lemma}[theorem]{Twisting lemma}
\newtheorem{lemma}[theorem]{Lemma}       
\newtheorem{proposition}[theorem]{Proposition}  

\newtheorem{question}[theorem]{Question}
}
{\theoremstyle{remark}
\newtheorem{definition}[theorem]{Definition}      
\newtheorem{remark}[theorem]{Remark}   

}

\def\cm{\hbox{\hbox{\rm C}\kern-5pt{\raise 1pt\hbox{$|$}}}}

\def\lhfl#1#2{\smash{\mathop{\hbox to 12mm{\leftarrowfill}}
\limits^{#1}_{#2}}}
\def\rhfl#1#2{\smash{\mathop{\hbox to 12mm{\rightarrowfill}}
\limits^{#1}_{#2}}}

\def\build#1_#2^#3{\mathrel{
\mathop{\kern 0pt#1}\limits_{#2}^{#3}}}

\def\htrait#1#2{\smash{\mathop{\hbox to 12mm{\hrulefill}}
\limits^{#1}_{#2}}}

\def\sxbullet{{\raise 2pt\hbox{\bf .}}}
\numberwithin{equation}{section}

\begin{document}

\title{On parametric extensions over number fields}

\author{Fran\c cois Legrand}

\email{flegrand@post.tau.ac.il}

\address{School of Mathematical Sciences, Tel Aviv University, Ramat Aviv, Tel Aviv 6997801, Israel}

\address{Department of Mathematics and Computer Science, the Open University of Israel, Ra'anana 4353701, Israel}

\date{\today}

\maketitle

\begin{abstract}
Given a number field $F$, a finite group $G$ and an indeterminate $T$, {\it{a $G$-parametric extension over $F$}} is a finite Galois extension $E/F(T)$ with Galois group $G$ and $E/F$ regular that has all the Galois extensions of $F$ with Galois group $G$ among its specializations. We are mainly interested in producing non-$G$-parametric extensions, which relates to classical questions in inverse Galois theory like the Beckmann-Black problem. Building on a strategy developed in previous papers, we show that there exists at least one non-$G$-parametric extension over $F$ for a given non-trivial finite group $G$ and a given number field $F$ under the sole necessary condition that $G$ occurs as the Galois group of a Galois extension $E/F(T)$ with $E/F$ regular.
\end{abstract}

\section{Introduction}

Given a number field $F$, the {\it{inverse Galois problem}} over $F$ asks whether every finite group $G$ occurs as the Galois group of a Galois extension of $F$. A classical way to obtain such an extension consists in introducing an indeterminate $T$ and in producing a Galois extension $E/F(T)$ with the same Galois group and $E/F$ {\it{regular}}\footnote{{\it{i.e.}}, $E\cap \overline{\Qq} = F$.}: from the {\it{Hilbert irreducibility theorem}}, the extension $E/F(T)$ has infinitely many linearly disjoint {\it{specializations}} with Galois group $G$ (if $G$ is not trivial). We refer to $\S$2.1 for basic terminology.
 
Following recent works \cite[\S4]{Leg16c} \cite{Leg15}, we are interested in the present paper in finite Galois extensions $E/F(T)$ with  $E/F$ regular - from now on, say for short that the extension $E/F(T)$ is an ``$F$-regular Galois extension" - that have all the Galois extensions of $F$ with Galois group $G$ among their specializations. More precisely, let us recall the following definition.

\begin{definition}
A finite $F$-regular Galois extension $E/F(T)$ with Galois group $G$ is {\it{$G$-parametric over $F$}} if every Galois extension of $F$ with Galois group $G$ occurs as a specialization of $E/F(T)$.
\end{definition}

Parametric extensions have been introduced with the aim of a better understanding of the {\it{Beckmann-Black problem}} which asks whether the specialization process to solve the inverse Galois problem is optimal. Namely, recall that the Beckmann-Black problem, for the finite group $G$ over the number field $F$, asks whether every Galois extension $L/F$ with Galois group $G$ is a specialization of some $F$-regular Galois extension $E_L/F(T)$ (possibly depending on $L/F$) with Galois group $G$. Although no counter-example is known and only a few positive results have been proved (see {\it{e.g.}} \cite[Theorem 2.2]{Deb01b} for more details), it may be expected that the Beckmann-Black problem fails in general over number fields. However, no line of attack seems to be known and a disproof is probably out of reach at the moment.

Actually, the answer to the following weaker question on parametric extensions seems to be unavailable in the literature. Say that a finite group $G$ is a ``regular Galois group over (the given number field) $F$" if $G$ occurs as the Galois group of an $F$-regular Galois extension of $F(T)$.

\begin{question} \label{ques}
Does there exist a regular Galois group $G$ over $F$ such that no $F$-regular Galois extension of $F(T)$ with Galois group $G$ is $G$-parametric over $F$?
\end{question}

\noindent
The existence of such a finite group $G$ would be a first step towards a counter-example to the Beckmann-Black problem over the number field $F$. However, although we may expect the answer to be negative almost always, deciding whether a given $F$-regular Galois extension of $F(T)$ with Galois group $G$ is $G$-parametric over $F$ or not is a difficult problem in general (even in the easiest case $G=\Zz/2\Zz$) and only a few non-parametric extensions are available in the literature. In particular, finding a group $G$ as in Question \ref{ques} seems to be difficult as well.

In \cite[\S4]{Leg16c} and \cite{Leg15}, we offer a systematic approach to produce $F$-regular Galois extensions $E/F(T)$ with Galois group $G$ which are not $G$-parametric over $F$. It consists in introducing another $F$-regular Galois extension $E'/F(T)$ with Galois group $G$ and in giving criteria ensuring that some specializations of $E'/F(T)$ with Galois group $G$ are not specializations of $E/F(T)$. Examples with specific finite groups $G$ such as abelian groups, symmetric and alternating groups, non-abelian simple groups, {\it{etc.}} are then given, under some natural necessary conditions. For example, an obvious obstruction to the existence of at least one non-$G$-parametric extension over $F$ is that $G$ is not a regular Galois group over $F$.

Building on this strategy, we show in this paper that the latter obstruction is the only one to the existence of a non-$G$-parametric extension over $F$ with Galois group $G$.

\begin{theorem} \label{thm main}
Let $G$ be a non-trivial finite group and $F$ a number field. Assume that $G$ is a regular Galois group over $F$. Then there exists at least one non-$G$-parametric extension over $F$ with Galois group $G$.
\end{theorem}

\noindent
Actually, from any $F$-regular Galois extension $E/F(T)$ with Galois group $G$ satisfying some mild assumptions on its set of branch points, we derive a sequence $(E_k/F(T))_k$ of $F$-regular realizations of $G$ such that infinitely many linearly disjoint specializations of $E/F(T)$ with Galois group $G$ are not specializations of $E_k/F(T)$. See Theorem \ref{thm main 2}.

The paper is organized as follows. In \S2, we recall some material used in the sequel. In \S3, we prove Theorem \ref{thm main} under an auxiliary result on {\it{prime divisors of polynomials}} (Definition \ref{pdop}) that has its own interest; see Proposition \ref{proposition 3}. Proposition \ref{proposition 3} is proved in \S4. Finally, in \S5, we make related previous results from \cite{Leg15} more precise thanks to a group theoretic argument communicated to us by Reiter.

\vspace{2mm}

{\bf{Acknowledgments.}} This work was motivated by a visit of the author in Universit\"at Bayreuth. The author is then indebted with Stefan Reiter for Lemma \ref{Reiter} and would like to thank the Zahlentheorie team for hospitality and financial support. The author also wishes to thank Lior Bary-Soroker, Pierre D\`ebes, Danny Neftin and Jack Sonn for helpful discussions, as well as the anonymous referee for suggesting a simpler proof of Proposition \ref{proposition 3}. This research is partially supported by the Israel Science Foundation (grants No. 40/14 and No. 696/13).

\section{Basics}

For this section, let $F$ be a number field.

\subsection{Specializations of finite Galois extensions of $F(T)$}

Given an indeterminate $T$, let $E/F(T)$ be a finite Galois extension with Galois group $G$ and $E/F$ {\it{regular}} ({\it{i.e.}}, $E \cap \overline{\Qq}=F$). From now on, say for short that $E/F(T)$ is an ``$F$-regular Galois extension".

Recall that a point $t_0 \in \mathbb{P}^1(\overline{\Qq})$ is {\it{a branch point of $E/F(T)$}} if the prime ideal $(T-t_0) \overline{\Qq}[T-t_0]$ \footnote{Replace $T-t_0$ by $1/T$ if $t_0 = \infty$.} ramifies in the integral closure of $\overline{\Qq}[T-t_0]$ in the {\it{compositum}} of $E$ and $\overline{\Qq}(T)$ (in a fixed algebraic closure of $F(T)$). The extension $E/F(T)$ has only finitely many branch points.

Given a point $t_0 \in \mathbb{P}^1(F)$, not a branch point, the residue extension of $E/F(T)$ at a prime ideal $\mathcal{P}$ lying over $(T-t_0) {F}[T-t_0]$ is denoted by ${E}_{t_0}/F$ and called {\it{the specialization of ${E}/F(T)$ at $t_0$}}. It does not depend on the choice of the prime $\mathcal{P}$  lying over $(T-t_0) {F}[T-t_0]$ as ${E}/F(T)$ is Galois. The extension $E_{t_0}/F$ is Galois with Galois group a subgroup of $G$, namely the decomposition group of ${E}/F(T)$ at $\mathcal{P}$.

\subsection{Prime divisors of polynomials}

Denote the integral closure of $\Zz$ in $F$ by $O_F$. Let $P(T) \in O_F[T]$ be a non-constant monic polynomial. 

\begin{definition} \label{pdop}
Say that a non-zero prime ideal $\mathcal{P}$ of $O_F$ is {\it{a prime divisor of $P(T)$}} if the reduction of $P(T)$ modulo $\mathcal{P}$ has a root in the residue field $O_F/\mathcal{P}$.
\end{definition}

The following lemma will be used on several occasions in the sequel. Denote the roots of $P(T)$ by $t_1,\dots,t_r$. Given an integer $k \geq 1$ and an index $j \in \{1,\dots,r\}$, let $\sqrt[k]{t_j}$ be a $k$-th root of $t_j$. Finally, let $L_k$ be the splitting field of $P(T^k)$ over $F$ and $\zeta_k$ a primitive $k$-th root of unity.

\begin{lemma} \label{Tchebotarev}
The following three conditions are equivalent:

\vspace{0.5mm}

\noindent
{\rm{(1)}} $\bigcup_{j=1}^r \bigcup_{l=0}^{k-1} {\rm{Gal}}(L_k/F(\zeta_k^l \sqrt[k]{t_j})) \not= \bigcup_{j=1}^r {\rm{Gal}}(L_k/F({t_j}))$,

\vspace{0.5mm}

\noindent
{\rm{(2)}} there exists a set $\mathcal{S}$ of non-zero prime ideals of $O_F$ that has positive density and such that each prime ideal $\mathcal{P}$ in $\mathcal{S}$ is a prime divisor of $P(T)$ but not of $P(T^k)$,

\vspace{0.5mm}

\noindent
{\rm{(3)}} there exist infinitely many non-zero prime ideals of $O_F$ each of which is a prime divisor of $P(T)$ but not of $P(T^k)$.
\end{lemma}

\begin{proof}
We may assume that $P(T)$ is separable. If $P(0)=0$, then (1), (2) and (3) fail. From now on, we assume that $P(0) \not=0$. In particular, $P(T^k)$ is separable.

First, assume that (1) holds, {\it{i.e.}}, there exists some $\sigma$ in
$$\bigcup_{j=1}^{r} {\rm{Gal}}(L_k/F(t_j)) \, \setminus \, \bigcup_{j=1}^{r} \bigcup_{l=0}^{k-1} {\rm{Gal}}(L_k/F(\zeta_k^l \sqrt[k]{t_j})).$$
By the Tchebotarev density theorem, there exists a positive density set $\mathcal{S}$ of primes $\mathcal{P}$ of $O_F$ such that the associated Frobenius in $L_k/F$ is conjugate to $\sigma$. As $\sigma$ fixes no root of $P(T^k)$, such a $\mathcal{P}$ is not a prime divisor of $P(T^k)$ (up to finitely many). Denote the splitting field of $P(T)$ over $F$ by $L_1$. Then the Frobenius associated with $\mathcal{P}$ in $L_1/F$ is the restriction to $L_1$ of the one in $L_k/F$. As $\sigma$ fixes a root of $P(T)$, $\mathcal{P}$ is a prime divisor of $P(T)$ (up to finitely many), as needed for (2).

As implication (2) $\Rightarrow$ (3) is obvious, it remains to prove implication (3) $\Rightarrow$ (1). To do this, assume that (1) does not hold. Let $\mathcal{P}$ be a non-zero prime ideal of $O_F$ that is a prime divisor of $P(T)$ and that is unramified in $L_k/F$. Denote the associated Frobenius in $L_k/F$ by $\sigma$. As $\mathcal{P}$ is a prime divisor of $P(T)$ and $\mathcal{P}$ does not ramify in $L_1/F$, the associated Frobenius in $L_1/F$ fixes a root of $P(T)$ (up to finitely many). Since this Frobenius is the restriction of $\sigma$ to $L_1$, we get that $\sigma$ fixes a root of $P(T)$. As condition (1) fails, $\sigma$ fixes a root of $P(T^k)$ as well. Hence $\mathcal{P}$ is a prime divisor of $P(T^k)$ (up to finitely many). Then (3) does not hold either, thus ending the proof.
\end{proof}

\section{Proof of Theorem \ref{thm main}}

The aim of this section consists in proving Theorem \ref{thm main 2} below whose Theorem \ref{thm main} is a straightforward application.

\subsection{Statement of Theorem \ref{thm main 2}} 

Let $F$ be a number field, $O_F$ the integral closure of $\Zz$ in $F$ and $G$ a non-trivial finite group that is a regular Galois group over $F$ ({\it{i.e.}}, $G$ occurs as the Galois group of an $F$-regular Galois extension of $F(T)$). 

Given an indeterminate $T$, let $E/F(T)$ be an $F$-regular Galois extension with Galois group $G$, branch points $t_1,\dots,t_r$ and such that the following two conditions hold\footnote{These two conditions hold up to applying a suitable change of variable.}:

\vspace{0.5mm}

\noindent
(bp-1) $\{0,1,\infty\} \cap \{t_1,\dots,t_r\} = \emptyset$,

\vspace{0.5mm}

\noindent
(bp-2) $t_1,\dots,t_r$ all are integral over $O_F$.

\begin{theorem} \label{thm main 2}
There exists a sequence of $F$-regular Galois extensions $E_k/F(T)$, $k \in \mathbb{N} \setminus \{0\}$ (depending on $E/F(T)$), with Galois group $G$ and that satisfies the following conclusion.

\noindent
For each finite extension $F'/F$, there exist infinitely many positive integers $k$ (depending on $F'$) such that the extension $E_kF'/F'(T)$ satisfies the following condition:

\vspace{1mm}

\noindent
{\rm{(non-$G$-parametricity)}} {\it{there exist infinitely many linearly disjoint Galois extensions of $F'$ with Galois group $G$ each of which is not a specialization of $E_kF'/F'(T)$.}}

\vspace{1mm}

\noindent
In particular, the extension $E_kF'/F'(T)$ is not $G$-parametric over $F'$. Furthermore, these Galois extensions of $F'$ with Galois group $G$ may be produced by specializing the extension $EF'/F'(T)$.
\end{theorem}

\begin{remark} \label{rk 4.2}
(1) As a classical consequence of the Riemann existence theorem, every finite group $G$ is a regular Galois group over some num-
ber field $F_G$, and then over every number field $F'$ containing $F_G$. Hence Theorem \ref{thm main 2} provides the following statement.

\vspace{1mm}

\noindent
{\it{Let $G$ be a non-trivial finite group. Then there exists some number field $F_G$ that satisfies the following property. For each number field $F'$ containing $F_G$, there exists an $F'$-regular Galois extension of $F'(T)$ with Galois group $G$  which satisfies the {\rm{(non-$G$-parametricity)}} condition. Moreover, one can take $F_G$ equal to a given number field $F$ if and only if $G$ is a regular Galois group over $F$.}}

\vspace{1.5mm}

\noindent
(2) As explained in \S3.2.4 below, we are not able to remove the dependence on the number field $F'$ containing $F$ in the set of all suitable positive integers $k$. In particular, the proof provides no integer $k$ such that the extension $E_kF'/F'(T)$ satisfies the {\rm{(non-$G$-parametricity)}} condition for each finite extension $F'/F$. See Proposition \ref{prop main precise} for a result with such a geometric conclusion.
\end{remark}

\subsection{Proof of Theorem \ref{thm main 2}}

We break the proof into four parts.

\subsubsection{Notation}

Given a positive integer $k$ and $j \in \{1,\dots,r\}$, let $\sqrt[k]{t_j}$ be a $k$-th root of $t_j$. Let $F'/F$ be a finite extension and $O_{F'}$ the integral closure of $\Zz$ in $F'$.

By condition (bp-1), one may consider the polynomial $$P_E(T) := \prod_{j=1}^r (T-t_j).$$ By condition (bp-2), the monic separable polynomial $P_E(T)$ has coefficients in $O_{F}$.

\subsubsection{Two lemmas}

Fix a positive integer $k$.

First, we derive from the extension $E/F(T)$ an $F$-regular Galois extension of $F(T)$ with group $G$ and specified set of branch points.

\begin{lemma} \label{lemma 1}
There exists an $F$-regular Galois extension of $F(T)$ with Galois group $G$ and whose branch points are exactly the $k$-th roots of those of $E/F(T)$.
\end{lemma}

\begin{proof}
The proof below follows part of an argument of D\`ebes and Zannier given in the proof of \cite[Proposition 5.2]{DW08}. Let $P(T,Y) \in F[T][Y]$ be the irreducible polynomial over $F(T)$ of some primitive element of $E$ over $F(T)$, assumed to be integral over $F[T]$. The polynomial $P(T,Y)$ is absolutely irreducible (as $E/F(T)$ is $F$-regular) and, as 0 is not a branch point (condition (bp-1)), it has a root in $\overline{\Qq}((T))$. By \cite[Lemma 0.1]{Deb92}, the polynomial $P_k(T,Y):=P(T^k,Y)$ is absolutely irreducible. Denote the field generated by one root of $P_k(T,Y)$ over $F(T)$ by $E_k$. The extension $E_k/F(T)$ is $F$-regular (as $P_k(T,Y)$ is absolutely irreducible) and has degree equal to the order of $G$. Denote the Galois closure of $E_k/F(T)$ by $\widehat{E_k}/F(T)$ and the Galois group of $\widehat{E_k}/F(T)$ by $H_k$. By the Hilbert irreducibility theorem, there are infinitely many $t_0 \in F$ such that the specialization $(\widehat{E_k})_{t_0}/F$ of $\widehat{E_k}/F(T)$ at $t_0$ has Galois group $H_k$. For all but finitely many $t_0 \in F$, the field $(\widehat{E_k})_{t_0}$ is the splitting field over $F$ of the polynomial $P_k(t_0,Y)=P(t_0^k,Y)$, which is in turn the field $E_{t_0^k}$. Hence there is a specialization of $E/F(T)$ with Galois group $H_k$. In particular, $H_k$ is a subgroup of $G$. As the order of $G$ divides the order of $H_k$, we get $G=H_k$. Hence $E_k/F(T)$ is an $F$-regular Galois extension with Galois group $G$. By construction, the branch points of $E_k/F(T)$ lying in $\overline{\Qq} \setminus \{0\}$ are the $k$-th roots of those of $E/F(T)$. As neither 0 nor $\infty$ is a branch point of $E/F(T)$ (condition (bp-1)), the same is true of $E_k/F(T)$, thus ending the proof.
\end{proof}

Let $E_k/F(T)$ be an $F$-regular Galois extension with Galois group $G$ and whose branch points are exactly the $k$-th roots of those of $E/F(T)$.

Next, we apply a previous criterion from \cite{Leg16c} for the extension $E_kF'/F'(T)$ to satisfy the (non-$G$-parametricity) condition.

\begin{lemma} \label{lemma 2}
Assume that there exist infinitely many non-zero prime ideals of $O_{F'}$ each of which is a prime divisor of $P_E(T)$ but not of $P_E(T^k)$ (considered as polynomials with coefficients in $F'$). Then the extension $E_kF'/F'(T)$ satisfies the {\rm{(non-$G$-parametricity)}} condition. Moreover, the Galois extensions of $F'$ with Galois group $G$ appearing in the {\rm{(non-$G$-parametricity)}} condition may be produced by specializing the extension $EF'/F'(T)$.
\end{lemma}

\begin{proof}
Given an algebraic number $t \not=0$, denote the irreducible polynomial of $t$ over $F'$ by $m_{t}(T)$. Consider the following four polynomials:
$$m_{EF'}(T) =\prod_{j=1}^r m_{t_j}(T),$$
$$m_{EF'}^*(T) = \prod_{j=1}^r m_{1/t_j}(T),$$
$$m_{E_kF'}(T) = \prod_{j=1}^r \prod_{l=0}^{k-1} m_{e^{2 i \pi l/k} \sqrt[k]{t_j}}(T),$$
$$m_{E_kF'}^*(T) = \prod_{j=1}^r \prod_{l=0}^{k-1} m_{1/(e^{2 i \pi l/k} \sqrt[k]{t_j})}(T).$$
By \cite[Theorem 4.2]{Leg16c} and since the branch points of the extension $E_k/F(T)$ are the $k$-th roots of those of $E/F(T)$, it suffices to prove that there exist infinitely many non-zero prime ideals of $O_{F'}$ each of which is a prime divisor of $m_{EF'}(T) \cdot m_{EF'}^*(T)$ but not of $m_{E_kF'}(T) \cdot m_{E_kF'}^*(T)$.

As $\infty$ is not a branch point of $EF'/F'(T)$ (condition (bp-1)), one may apply \cite[Remark 3.11]{Leg16c} to get that $m_{EF'}(T) \cdot m_{EF'}^*(T)$ and $m_{EF'}(T)$ have the same prime divisors (up to finitely many). Since the polynomials $m_{EF'}(T)$ and $P_{E}(T)$ have the same prime divisors, we get that $m_{EF'}(T) \cdot m_{EF'}^*(T)$ and $P_{E}(T)$ have the same prime divisors (up to finitely many). By the same argument, every prime divisor of $m_{E_kF'}(T) \cdot m_{E_kF'}^*(T)$ is a prime divisor of $P_{E}(T^k)$ (up to finitely many). Then, from the assumption in the statement, there exist infinitely many non-zero prime ideals of $O_{F'}$ each of which is a prime divisor of $m_{EF'}(T) \cdot m_{EF'}^*(T)$ but not of $m_{E_kF'}(T) \cdot m_{E_kF'}^*(T)$, as needed.
\end{proof}

\subsubsection{A number theoretical result}

Now, we need the following number theoretical result to ensure that the assumption of Lemma \ref{lemma 2} holds.

\begin{proposition} \label{proposition 3}
Given a monic separable polynomial $P(T) \in O_F[T]$ such that $P(0) \not=0$ and $P(1) \not=0$, there is an infinite set $S$ of integers $k \geq 1$ such that, for each $k \in S$, there are infinitely many prime ideals of $O_{F}$ each of which is a prime divisor of $P(T)$ but not of $P(T^k)$.
\end{proposition}

\begin{remark} \label{remark F'}
(1) If either 0 or 1 is a root of $P(T)$, then the conclusion of Proposition \ref{proposition 3} clearly fails.

\vspace{0.5mm}

\noindent
{\rm{(2)}} The set $S$ depends on the polynomial $P(T)$ and this dependence cannot be removed. Indeed, given an integer $k \geq 1$, all non-zero prime ideals of $O_{F}$ are prime divisors of $P(T)$ and $P(T^k)$ if $P(T)= T-2^k$.

\vspace{0.5mm}

\noindent
{\rm{(3)}} Similarly, the set $S$ depends on the number field $F$ and this dependence cannot be removed. Indeed, given a number field $F'$ containing $F$ and an integer $k \geq 1$, all but finitely many prime ideals of $O_{F'}$ are prime divisors of $P(T)$ and $P(T^k)$ if $F'$ contains a root of $P(T^k)$.
\end{remark}

Proposition \ref{proposition 3} is proved in \S4.

\subsubsection{Conclusion}

As already said, the monic separable polynomial $P_E(T)$ has coefficients in $O_{F'}$. Moreover, by condition (bp-1), one has $P_E(0) \not=0$ and $P_E(1) \not=0$. Then, by Proposition \ref{proposition 3} (applied over $F'$), there exists an infinite set $S$ of positive integers $k$ (depending on $F'$; see Remark \ref{remark F'}) such that, for each $k \in S$, there exist infinitely many non-zero prime ideals of $O_{F'}$ each of which is a prime divisor of $P_E(T)$ but not of $P_E(T^k)$. It then remains to apply Lemma \ref{lemma 2} to conclude.

\section{Proof of Proposition \ref{proposition 3}}

This section is organized as follows. In \S4.1, we state Proposition \ref{prop partiel} which is Proposition \ref{proposition 3} for polynomials whose roots all are in the base number field. Next, we explain in \S4.2 how deducing Proposition \ref{proposition 3} from Proposition \ref{prop partiel}. Finally, Proposition \ref{prop partiel} is proved in \S4.3.

\subsection{Statement of Proposition \ref{prop partiel}}

\begin{proposition} \label{prop partiel}
Given a number field $F$, let $P(T) \in O_F[T]$ \footnote{As before, $O_F$ denotes the integral closure of $\Zz$ in $F$.} be a monic separable polynomial whose roots all are in $F \setminus \{0,1\}$. Then there exist infinitely many positive integers $k$ such that the Galois group of $P(T^k)$ over $F$ has an element fixing no root of $P(T^k)$.
\end{proposition}

\subsection{Proof of Proposition \ref{proposition 3} under Proposition \ref{prop partiel}}

Let $F$ be a number field and $P(T) \in O_F[T]$ a monic separable polynomial such that $P(0) \not=0$ and $P(1) \not=0$. Denote the roots of $P(T)$ by $t_1, \dots, t_r$ and the splitting field of $P(T)$ over $F$ by $L$. By Proposition \ref{prop partiel}, there exist infinitely many positive integers $k$ such that
$${\rm{Gal}}(L_k/L) \setminus \bigcup_{j=1}^r \bigcup_{l=0}^{k-1} {\rm{Gal}}(L_k/L(\zeta_k^l \sqrt[k]{t_j}))$$
contains some $\sigma_k$,
where $L_k$ is the splitting field over $L$ of $P(T^k)$, $\zeta_k$ is a primitive $k$-th root of unity and $\sqrt[k]{t_j}$ is a given $k$-th root of $t_j$ ($j=1,\dots,r$). For each positive integer $k$, the splitting field of $P(T^k)$ over $F$ is equal to $L_k$. Then $\sigma_k$ lies in
$$\bigcup_{j=1}^r {\rm{Gal}}(L_k/F({t_j})) \setminus \bigcup_{j=1}^r \bigcup_{l=0}^{k-1} {\rm{Gal}}(L_k/F(\zeta_k^l \sqrt[k]{t_j})).$$
It then remains to use implication (1) $\Rightarrow$ (3) of Lemma \ref{Tchebotarev} to conclude.

\subsection{Proof of Proposition \ref{prop partiel}}
 
We proceed by induction on the degree of the polynomial $P(T)$.

\subsubsection{The case where $P(T)$ has degree 1}

Let $F$ be a number field and $t \in O_F \setminus \{0,1\}$. The conclusion of Proposition \ref{prop partiel} for the polynomial $T^k-t$ easily follows from Lemmas \ref{unity} and \ref{not unity} below\footnote{In the case where $t$ is not a root of unity, one makes use of the following classical lemma: if a given finite group $G$ acts transitively on a given finite set $X$ with cardinality at least 2, then there exists $g \in G$ such that $g . x = x$ for no $x \in X$.}.

\begin{lemma} \label{unity}
Assume that $t$ is a root of unity. For each number field $L$ containing $F$, there exist infinitely many integers $k \geq 1$ such that the Galois group of $T^k-t$ over $L$ is not trivial and each non-trivial element of the Galois group of $T^k-t$ over $F$ fixes no root of this polynomial.
\end{lemma}

\begin{proof}
Assume that $t$ is a primitive $N$-th root of unity. Let $L$ be a number field containing $F$ and $k$ a positive integer whose prime factors all are prime factors of $N$. As $t \not=1$, one has $N \geq 2$ and there exist infinitely many such integers $k$. Assume that the Galois group of $T^k-t$ over $L$ is trivial. Then $L$ contains a primitive $k$-th root of unity, which cannot happen if $k$ is sufficiently large (depending on $L$). One may then assume that the Galois group of $T^k-t$ over $L$ is not trivial. In particular, the Galois group of $T^k-t$ over $F$ is not trivial either. Let $\sigma$ be a non-trivial element of the latter Galois group. Assume that $\sigma$ fixes at least one root of $T^k-t$. By the definition of $k$, each root of $T^k-t$ is a primitive $(Nk)$-th root of unity. This implies that $\sigma$ fixes each root of $T^k-t$, which cannot happen.
\end{proof}

\begin{lemma} \label{not unity}
Assume that $t$ is not a root of unity. Then $T^k-t$ is irreducible over $F$  for all but finitely many prime numbers $k$.
\end{lemma}

\begin{proof}
By the Capelli lemma (see {\it{e.g.}} \cite[Chapter VI, \S9, Theorem 9.1]{Lan02}), it suffices to show that, for all but finitely many prime numbers $k$, $t$ is not a $k$-th power in $F$. Denote the absolute logarithmic Weil height on $\overline{\Qq}$ by $h$. Assume that there exist infinitely many integers $k \geq 1$ such that there exists $x_k \in F$ satisfying $t=x_k^k$. One then has
$$h(t) = h(x_k^k) =k \cdot {h(x_k)}.$$
As $t$ is not a root of unity and $t \not=0$, one has $h(t) \not=0$. Hence $F$ contains infinitely many elements each of which has height bounded by $h(t)$, which cannot happen \cite[Theorem 1]{Nor49}. 
\end{proof}

\subsubsection{End of the proof of Proposition \ref{prop partiel}}

Let $r$ be a positive integer. Assume that the following condition holds:

\vspace{1mm}

\noindent
(H) {\it{For each number field $F$ and each monic separable degree $r$ polynomial $P(T) \in O_F[T]$ whose roots all are in $F \setminus \{0,1\}$, there exist infinitely many  positive integers $k$ such that the Galois group of $P(T^k)$ over $F$ has an element fixing no root of $P(T^k)$.}}

\vspace{1mm}

\noindent
Let $F$ be a number field and let $P(T) \in O_F[T]$ be a monic separable degree $r+1$ polynomial whose roots all are in $F \setminus \{0,1\}$. Denote the roots of $P(T)$ by $t_1, \dots, t_r, t_{r+1}$. By condition (H), there exists an integer $k_0 \geq 1$ such that the Galois group of $ (T^{k_0} - t_1) \cdots (T^{k_0}-t_{r})$ over $F$ has an element $\tau$ fixing no root of this polynomial. Denote the splitting field of $ (T^{k_0} - t_1) \cdots (T^{k_0}-t_{r})$ over $F$ by $L$. 

\vspace{2mm}

\noindent
(a) Assume that $t_{r+1}$ is a root of unity. By Lemma \ref{unity}, there is an integer $k_1 \geq 1$ such that the Galois group of $T^{k_1} - t_{r+1}$ over $L$ is not trivial and every non-trivial element of the Galois group of $T^{k_1}- t_{r+1}$ over $F$ fixes no root of this polynomial. Let $\sigma$ be a non-trivial element of the former Galois group. Denote the splitting field of $T^{k_1} - t_{r+1}$ over $L$ by $M$ and let $\hat{\tau} \in {\rm{Gal}}(M/F)$ be a prolongation of $\tau$ to $M$.

First, assume that $\hat{\tau}$ fixes no root of $T^{k_1} - t_{r+1}$. Then $\hat{\tau}$ is an element of  the Galois group of $ (T^{k_0} - t_1) \cdots (T^{k_0}-t_{r}) \cdot (T^{k_1} - t_{r+1})$ over $F$ fixing no root of this polynomial. Given a positive multiple $k$ of $k_0$ and $k_1$, every prolongation of $\hat{\tau}$ to the splitting field $M_k$ over $F$ of $P(T^k)$ is an element of ${\rm{Gal}}(M_k/F)$ fixing no root of this polynomial. Hence the desired conclusion holds.

Now, assume that $\hat{\tau}$ fixes a root of $T^{k_1} - t_{r+1}$. By the definition of $k_1$, $\hat{\tau}$ fixes each root of $T^{k_1} - t_{r+1}$. Consider the element $\sigma \hat{\tau}$ of ${\rm{Gal}}(M/F)$. If $x$ denotes any $k_0$-th root of $t_1, \dots, t_r$, then $\hat{\tau}(x)$ still is a $k_0$-th root of $t_1, \dots, t_r$. By the definition of $\sigma$, one then has $\sigma \hat{\tau} (x) = \hat{\tau}(x)$, which is not equal to $x$ by the definition of $\hat{\tau}$. If $x$ denotes any $k_1$-th root of $t_{r+1}$, then, by the above, one has 
$\sigma \hat{\tau}(x) = \sigma(x)$, which is not equal to $x$ by the definition of $\sigma$. Hence $\sigma \hat{\tau}$ is an element of the Galois group of $ (T^{k_0} - t_1) \cdots (T^{k_0}-t_{r}) \cdot (T^{k_1} - t_{r+1})$ over $F$ fixing no root of this polynomial. As before, the desired conclusion easily follows.

\vspace{2mm}

\noindent
(b) Assume that $t_{r+1}$ is not a root of unity. By Lemma \ref{not unity}, $T^{k_1}-t_{r+1}$ is irreducible over $L$ for some prime $k_1$. As before, denote the splitting field of $T^{k_1} - t_{r+1}$ over $L$ by $M$ and let $\hat{\tau} \in {\rm{Gal}}(M/F)$ be a prolongation of $\tau$ to $M$. Let $\sigma$ be an element of ${\rm{Gal}}(M/L)$ fixing no root of $T^{k_1}-t_{r+1}$. If $\hat{\tau}$ fixes no root of $T^{k_1} - t_{r+1}$, then one gets the desired conclusion as in (a) above. We may then assume that $\hat{\tau}(\alpha)=\alpha$ for some root $\alpha$ of $T^{k_1}-t_{r+1}$. Let $\zeta$ be a primitive $k_1$-th root of unity. Up to making the prime number $k_1$ sufficiently large (depending on $L$), we may assume that $L$ and $\Qq(\zeta)$ are linearly disjoint over $\Qq$. Then $L(\zeta)/L$ has degree $k_1-1$. Hence $L(\alpha)$ and $L(\zeta)$ are linearly disjoint over $L$ (as $L(\alpha)/L$ has degree $k_1$). As a consequence, the Galois group ${\rm{Gal}}(M/L(\alpha))$ is generated by some element $b$ satisfying $b(\alpha)= \alpha$ and $b(\zeta)=\zeta^e$ for some $e \in \mathbb{N}$. Consider the restriction $w$ of $\hat{\tau}$ to $F(\zeta)$. One has $w=c^m$ for some integer $m$, where $c$ is the generator of ${\rm{Gal}}(F(\zeta)/F)$ defined by $c(\zeta)=\zeta^e$. Then $\hat{\tau} b^{-m}$ is a prolongation of $\tau$ to $M$ which fixes each root of $T^{k_1}-t_{r+1}$ (as $\hat{\tau}(\alpha)=\alpha$). As in (a) above, one shows that $\sigma \hat{\tau} b^{-m}$ is an element of the Galois group over $F$ of $ (T^{k_0} - t_1) \cdots (T^{k_0}-t_{r}) \cdot (T^{k_1} - t_{r+1})$ fixing no root of this polynomial, thus ending the proof.

\section{A geometric variant}

The aim of this section is Proposition \ref{prop main precise} below which makes \cite[Corollary 5.2]{Leg15} more precise (this result is recalled as Lemma \ref{lemma 0} below).

\subsection{Statement of Proposition \ref{prop main precise}}

\begin{proposition} \label{prop main precise}
Let $G$ be a non-trivial finite group, not a cyclic $p$-group. Then there exist a number field $F_G$ and an $F_G$-regular Galois extension $E/F_G(T)$ with Galois group $G$ such that the following holds:

\vspace{1mm}

\noindent
{\rm{(geometric non-$G$-parametricity)}} {\it{for every finite extension $F'/F_G$, there exist infinitely many linearly disjoint Galois extensions of $F'$ with Galois group $G$ each of which is not a specialization of $EF'/F'(T)$ \footnote{As in the (non-$G$-parametricity) condition, the realizations of $G$ whose existence is claimed may be produced by specialization.}.}}
\end{proposition}

\noindent
Unlike the result in part (1) of Remark \ref{rk 4.2}, it seems unclear whether a number field $F_G$ as in Proposition \ref{prop main precise} may be specified for a given group $G$ \footnote{{\it{i.e.}}, being a regular Galois group over a given number field $F$ might not be sufficient to take $F_G=F$.}. See \cite[\S7]{Leg15} where this is done in some specific cases.

\subsection{Proof of Proposition \ref{prop main precise}}

Let $G$ be a non-trivial finite group. 

First, recall the following result which is \cite[Corollary 5.2]{Leg15}.

\begin{lemma} \label{lemma 0}
There exist a number field $F_G$ and an $F_G$-regular Galois extension of $F_G(T)$ with Galois group $G$ which satisfies the {\rm{(geometric non-$G$-parametricity)}} condition if the following group theoretic condition holds.

\vspace{1mm}

\noindent
{\rm{(H2)}} There exists a set $\{C, C_1,\dots,C_r\}$ of non-trivial conjugacy classes of $G$ such that the elements of $C_1,\dots,C_r$ generate $G$ and the remaining conjugacy class $C$ is not in the set $\{C_1^a , \dots, C_r^a\, / \,   a \in \mathbb{N} \}$. 
\end{lemma}

Now, combine Lemmas \ref{lemma 0} and \ref{Reiter} below to get Proposition \ref{prop main precise}.

\begin{lemma} \label{Reiter}
Condition {\rm{(H2)}} fails if and only if $G$ is a cyclic $p$-group.
\end{lemma}

\begin{proof}[Proof of Lemma \ref{Reiter}] 
It is not hard to see that condition {\rm{(H2)}} fails if $G$ is a cyclic $p$-group. For the converse, we use the following argument due to Reiter. Assume that condition (H2) fails. Let $H$ be a maximal subgroup of $G$. If $H$ is not a normal subgroup of $G$, one has
\begin{equation} \label{eq 1}
G = \Big\langle \bigcup_{g \in G} \, gHg^{-1} \Big\rangle.
\end{equation}
As condition (H2) fails, \eqref{eq 1} provides $G = \bigcup_{g \in G} \, gHg^{-1},$ which cannot happen. Then each maximal subgroup of $G$ is a normal one. Hence $G$ is nilpotent, {\it{i.e.}}, $G$ is the product of its Sylow subgroups. Set 
\begin{equation} \label{eq 2}
G=P_1 \times \cdots \times P_s
\end{equation}
with $P_1,\dots, P_s$ the Sylow subgroups of $G$. By the Sylow theorems and as condition (H2) has been assumed to fail, \eqref{eq 2} provides
\begin{equation} \label{eq 3}
G= P_1 \cup \cdots \cup P_s.
\end{equation}
If $s \geq 2$, then, by taking cardinalities in \eqref{eq 2} and \eqref{eq 3}, we get 
$$\prod_{i=1}^s |P_i| < \sum_{i=1}^s |P_i|,$$
which cannot happen. Hence $s=1$ and $G$ is a $p$-group. 

Let $H_1$ and $H_2$ be two distinct maximal subgroups of $G$. Then
\begin{equation} \label{eq 4}
G= \langle H_1 \cup H_2 \rangle.
\end{equation}
As $H_1$ and $H_2$ are normal subgroups of $G$ and as condition (H2) fails, \eqref{eq 4} provides $G = \bigcup_{g \in G} g(H_1 \cup H_2)g^{-1}.$ Hence $G=H_1 \cup H_2$. In particular, one has $H_1 \subseteq H_2$ or $H_2 \subseteq H_1$, which cannot happen. Hence $G$ has only one maximal subgroup and is then cyclic, as needed.
\end{proof}

\subsection{A conjectural version of Proposition \ref{prop main precise}}

Recall that \cite{Leg15} also offers a conjectural version of \cite[Corollary 5.2]{Leg15}; see \cite[Corollary 5.3]{Leg15}. Below we provide a similar conjectural version of Proposition \ref{prop main precise} (which then makes \cite[Corollary 5.3]{Leg15} more precise). 

Namely, let $G$ be a non-trivial finite group. Assume that the following conjecture of Fried is satisfied\footnote{See Section III.1 of \url{http://www.math.uci.edu/~mfried/deflist-cov/RIGP.html} or \cite[\S5]{Leg15} for more details.}.

\vspace{2mm}

\noindent
{\bf{Conjecture (Fried).}} {\it{Each set $\{C_1,\dots,C_r\}$ of non-trivial conjugacy classes of $G$ that is rational and such that the elements of $C_1,\dots,C_r$ generate $G$ occurs as the inertia canonical conjugacy class set of some $\mathbb{Q}$-regular Galois extension of $\mathbb{Q}(T)$ with Galois group $G$.}}

\vspace{2mm}

Then, by combining Lemma \ref{Reiter} and \cite[Corollary 5.3]{Leg15}, Proposition \ref{prop main precise} holds with $F_G=\Qq$, {\it{i.e.}}, the following holds.

\begin{proposition} \label{prop}
Assume that $G$ is not a cyclic $p$-group. Then there exists a $\Qq$-regular Galois extension of $\Qq(T)$ with Galois group $G$ that satisfies the {\rm{(geometric non-$G$-parametricity)}} condition.
\end{proposition}

\subsection{Other base fields} 

We conclude this paper by noticing that similar statements can be given for other base fields. For example, by conjoining Lemma \ref{Reiter} and \cite[\S5.2]{Leg15}, we obtain the following counterpart of Proposition \ref{prop main precise} for rational function fields.

\begin{proposition} \label{alg clos}
Let $G$ be a non-trivial finite group, not a cyclic $p$-group, $\kappa$ an algebraically closed field of characteristic zero and $X$ an indeterminate such that $T$ is transcendental over $\kappa(X)$. Then, for some Galois extension $E/\overline{\Qq}(T)$ with group $G$, the extension $E \kappa(X)/\kappa(X)(T)$ satisfies the {\rm{(geometric non-$G$-parametricity)}} condition.
\end{proposition}

\bibliography{Biblio2}
\bibliographystyle{alpha}

\end{document}